\documentclass[10pt]{article}
\usepackage{amsbsy}
\usepackage{amsmath,amsfonts,amssymb,amsthm,mathrsfs,mathtools,stmaryrd,bm}
\usepackage{float}
\usepackage{doi,url}
\usepackage[numbers]{natbib}
\usepackage{etoolbox,letltxmacro,xifthen,ifdraft} 
\usepackage[dvipsnames]{xcolor}
\usepackage{amsmath,amssymb,enumitem,bbm,mathrsfs,amsthm,aliascnt,mathtools}
\usepackage{graphicx,subcaption}
\usepackage{hyperref} 
\hypersetup{final,hidelinks}
\usepackage[atend]{bookmark} 

\providecommand{\email}[1]{\href{mailto:#1}{\nolinkurl{#1}}}

\setlist[enumerate,1]{label={(\roman*)}}
\setlist[enumerate,2]{label={(\alph*)}}
\setlist[enumerate,3]{label={(\Roman*)}}

\newcommand{\newsstheorem}[2]{
  \newaliascnt{#1}{dummy}
  \newtheorem{#1}[#1]{#2}
  \aliascntresetthe{#1}
  \expandafter\def\csname #1autorefname\endcsname{#2}
}
\numberwithin{dummy}{section}
\theoremstyle{plain}
  \newsstheorem{theorem}{Theorem}
  \newsstheorem{proposition}{Proposition}
  \newsstheorem{corollary}{Corollary}
  \newsstheorem{lemma}{Lemma}
  \newsstheorem{definition}{Definition}
  \theoremstyle{definition}
  \newsstheorem{example}{Example}
  \newsstheorem{assumption}{Assumption}
\theoremstyle{remark}
  \newsstheorem{remark}{Remark}

\newenvironment{eqnarr*}{\begin{IEEEeqnarray*}{rCl}}{\end{IEEEeqnarray*}\ignorespacesafterend}

\newcommand\RR{\mathbb{R}}

\newcommand\PP{\mathbb{P}}

\newcommand\EE{\mathbb{E}}
\newcommand\NN{\mathbb{N}}

\newcommand\e{{\mathrm e}}

\makeatletter \newcommand\mathof[1]{{\operator@font#1}} \makeatother

\newcommand\dd{\mathof{d}}

\begin{document}

\title{How linear reinforcement  affects  \\
Donsker's Theorem for empirical processes}
\author{Jean Bertoin\footnote{Institute of Mathematics, University of Zurich, Switzerland, \texttt{jean.bertoin@math.uzh.ch}}  }
\date{\small Dedicated to the memory of Harry Kesten, \\ for the deep mathematics he gave us.}
\maketitle 
\thispagestyle{empty}

\begin{abstract} 
{A reinforcement algorithm introduced by H.A. Simon \cite{Simon} produces a sequence of uniform random variables with memory 
as follows. At each step, with a fixed probability $p\in(0,1)$, $\hat U_{n+1}$ is sampled uniformly from $\hat U_1, \ldots, \hat U_n$, and with complementary probability $1-p$, $\hat U_{n+1}$ is a new independent uniform variable. The Glivenko-Cantelli theorem remains valid  for the reinforced empirical measure, but not the Donsker theorem.
Specifically, we show that the sequence of empirical processes converges in law to a Brownian bridge only up to a constant factor when $p<1/2$, and  that a further rescaling is needed  when  $p>1/2$ and the limit is then a bridge with exchangeable increments and discontinuous paths. This is related to earlier limit theorems for correlated Bernoulli processes, the so-called elephant random walk, and more generally step reinforced random walks.
 }
\newline  \vskip 1mm
{\normalfont \bfseries Keywords:}
Donsker's Theorem, empirical process, linear reinforcement,  bridges with exchangeable increments.\newline
\vskip 1mm
{\normalfont \bfseries Mathematics Subject Classification:}  60F17; 62G30

\end{abstract}

\section{Introduction}\label{s:intro}

 A classical result of Glivenko and Cantelli in 1933 states that  the sequence of empirical distribution functions associated to  i.i.d. copies of some real random variable converges uniformly to its cumulative distribution function, almost surely. Nearly 20 years later, Donsker determined the asymptotic behavior of the fluctuations; let us recall his result.
Let  $U_1, U_2, \ldots$ be  i.i.d.  uniform random variables on $[0,1]$; then  the sequence of (uniform) empirical processes,
$$
G_n(x)=\frac{1}{\sqrt n} \sum_{i=1}^n\left( {\mathbf 1}_{U_i \leq x}-x\right), \qquad x\in [0,1],
$$
converges in distribution as $n\to \infty$ in the sense of Skorokhod towards a Brownian bridge  $(G(x))_{0\leq x \leq 1}$.
The purpose of the present work is to analyze how Donsker's Theorem is affected by an elementary random reinforcement algorithm that we shall now describe. 

Consider a sequence $\varepsilon_2, \varepsilon_3, \ldots$ of i.i.d. Bernoulli variables with fixed parameter $p\in(0,1)$. These variables determine when repetitions occur, in the sense that 
the $n$-th step of the algorithm is a repetition  if $\varepsilon_n=1$, and  an innovation if $\varepsilon_n=0$.
For every $n\geq 2$, 
let also $v(n)$ be a uniform random variable on $\{1, \ldots, n-1\}$ such that $v(2), v(3), \ldots$ are independent;
these variables specify which of the preceding items is copied when a repetition occurs.
More precisely, we set $\varepsilon_1=0$ for definitiveness and construct recursively a sequence of random variables $\hat U_1, \hat U_2, \ldots$ by setting
$$\hat U_n=\left\{
\begin{matrix}
\hat U_{v(n)}&\text{ if }\varepsilon_n=1, \\
U_{{\mathrm i}(n)}&\text{ if }\varepsilon_n=0,\\
\end{matrix}
\right.
$$
where $${\mathrm i}(n)=\sum_{j=1}^n(1-\varepsilon_j)\qquad\text{for }n\geq 1$$
denotes the total number of innovations after $n$ steps. 
We always assume without further mention that the sequences $(v(n))_{n\geq 2}$, $(\varepsilon_n)_{n\geq 2}$, 
and $(U_j)_{j\geq 1}$ are independent. 

This random algorithm  has been introduced  in 1955 by Herbert A. Simon \cite{Simon}, who singled out in this setting 
 a remarkable one-parameter family of power tail distributions on $\NN$ that  arise  in a variety of data.
 Nowadays,  Simon's algorithm 
 should be  viewed as a linear reinforcement procedure,
in the sense that, provided that 
 ${\mathrm i}(n)\geq j$  (i.e. the variable $U_j$ as already appeared at the $n$-th step of the algorithm), the probability that $U_j$ is repeated at the $(n+1)$-th step is proportional to the number of its previous occurrences. In this direction, we refer henceforth to the parameter $p$ of the Bernoulli variables $\varepsilon_n$ as the reinforcement parameter.

Obviously, each variable $\hat U_n$ 
has the uniform distribution on $[0,1]$; note however that the reinforced sequence $(\hat U_n)_{n\geq 1}$ is  clearly not stationary, and is not exchangeable or even partially exchangeable either. 
It is easy to show that nonetheless, the conclusion of the  Glivenko-Cantelli theorem is still valid in this framework:
\begin{proposition}\label{P:GlivCant} With probability one, it holds that
$$\lim_{n\to \infty} \sup_{0\leq x \leq 1}  \left| \frac{1}{n}  \sum_{i=1}^n\left( {\mathbf 1}_{\hat U_i\leq x}-x\right) \right|=0.$$
\end{proposition}

We are chiefly interested in  the empirical 
processes $\hat G_n$ associated to the reinforced sequence
$$\hat G_n(x)=\frac{1}{\sqrt n} \sum_{i=1}^n\left( {\mathbf 1}_{\hat U_i \leq x}-x\right), \qquad x\in [0,1].$$
Our main result shows that their asymptotic behavior as $n\to \infty$ exhibits a phase transition for the critical parameter $p_c=1/2$. Roughly speaking, when the reinforcement parameter $p$ is smaller than $1/2$, then 
the analog of Donsker's Theorem holds for $\hat G_n$,  except that the limit
is now only proportional to the Brownian bridge.  At criticality, i.e. for $p=1/2$, convergence in distribution to the Brownian bridge holds after an additional rescaling of $\hat G_n$ by a factor $1/\sqrt{\log n}$. Finally for $p>1/2$, $n^{-p+1/2}\hat G_n$ now converges in probability and its limit is described in terms of some bridge with exchangeable increments and discontinuous sample paths.

Here is a precise statement, where the needed background on bridges with exchangeable increments in the supercritical case $p>1/2$ is postponed to the next section. Recall that $G=(G(x))_{0\leq x \leq 1}$ denotes the standard Brownian bridge. We further write ${\mathbb D}$ for the space of c\`adl\`ag paths $\omega: [0,1]\to \RR$ endowed  with the Skorokhod topology
(see Chapter 3 in \cite{Billing} or Chapter VI in \cite{JS}). The notation
$ \Rightarrow$ is used to indicate convergence  in distribution of a sequence of processes in ${\mathbb D}$. 
\begin{theorem} \label{T1} The following convergences hold as $n\to \infty$:
\begin{enumerate}
\item[(i)] If $p<1/2$, then 
$$\hat G_n \ \Longrightarrow \ \frac{1}{\sqrt{1-2p}}\, G.$$
\item[(ii)] If $p=1/2$, then  
$$\frac{1}{\sqrt{\log n}}\, \hat G_n \ \Longrightarrow \ G.$$
\item[(iii)] If $p>1/2$, then 
$$\lim_{n\to \infty} n^{-p+1/2}\hat G_n = B^{(p)} \qquad \text{
in probability on }{\mathbb D},$$
 where $B^{(p)}=(B^{(p)}(x))_{0\leq x \leq 1}$ is the bridge with exchangeable increments described in the forthcoming Definition \ref{D1}.
\end{enumerate}
\end{theorem}

Our approach to Theorem \ref{T1} relies, at least in part,  on a natural interpretation of Simon's algorithm in terms of Bernoulli bond percolation on random recursive trees.
Specifically, we view  $\{1,2,\ldots,n\}$  as a set of vertices and each pair $(j,v(j))$ for $j=2, \ldots, n$ as edges; the resulting graph ${\mathbb T}_n$ is known as the random recursive tree of size $n$, see Section 1.3 and Chapter 7 in \cite{Drmota}.  We next delete each edge $(j,v(j))$ if and only if $\varepsilon_j=0$, in other words we perform a Bernoulli bond percolation with parameter $p$ on ${\mathbb T}_n$. The percolation clusters are then given by subsets of indices at which the same variable is repeated, namely $\{i\leq n: \hat U_i=U_j\}$ for $j=1, \ldots , {\mathrm i}(n)$. 

The sum of the squares of the cluster sizes
$${\mathcal S}^2(n)=\sum_{j\geq 1}N_j(n)^2, \qquad \text{with }N_j(n) =\#\{i\leq n: \hat U_i=U_j\},$$
lies at the heart of  the analysis of the reinforced empirical process $\hat G_n$. We shall see that its asymptotic behavior is given by
\begin{equation}\label{E:expt}
{\mathcal S}^2(n)\sim \left\{ 
\begin{matrix} n/(1-2p) &\text{ if }p<1/2,\\
n\log n &\text{ if }p=1/2,\\
n^{2p}R&\text{ if }p>1/2,\\
\end{matrix}
\right.
\end{equation}
where $R$ is some non-degenerate random variable. A rough explanation for the phase transition\footnote{Somehow, the fact that percolation on random recursive trees exhibits a  phase transition with critical parameter $p_c=1/2$ bears  a flavor similar  to Kesten's celebrated achievement \cite{kesten80}  for bond percolation
on the square lattice. However, this resemblance is purely coincidental and superficial.} in \eqref{E:expt} is that the main contribution to ${\mathcal S}^2(n)$ is due to a large number of microscopic clusters  in the sub-critical case $p<1/2$, and rather to a few mesoscopic clusters of size $\approx n^p$ in the super-critical case $p>1/2$. 
Even though  \eqref{E:expt} is not quite sufficient to establish Theorem \ref{T1}, it is nonetheless a major step for its proof.
More precisely,  we shall rely on general results due to Kallenberg \cite{Kal2} on the structure of processes with exchangeable increments and explicit criteria for the 
 weak convergence of sequences of the latter, and \eqref{E:expt} appears as a key element in this setting. 

The rest of this work is organized as follows. Section 2 is devoted to several preliminaries. We shall first present some key results due to Kallenberg on bridges with exchangeable increments and their canonical representations. We shall then  recall a limit theorem for the numbers of occurrences $N_j(n)$ which have been obtained in the framework of Bernoulli percolation on random recursive trees as well as the fundamental result of H.A. Simon about the frequency of microscopic clusters. Last, we shall compute  explicitly the average $\EE({\mathcal S}^2(n))$ using a simple recurrence identity and establish Proposition \ref{P:GlivCant} on our way. Theorem \ref{T1} is then proven in Section 3. Finally, in Section 4, we discuss some connections between Theorem \ref{T1} and closely related results in the literature on step-reinforced random walks, including correlated Bernoulli processes and the so-called elephant random walk.

\section{Preliminaries}
\subsection {Bridges with exchangeable increments}
This section is adapted from Kallenberg \cite{Kal2}, who rather uses the terminology  interchangeable instead of exchangeable, and whose results are given in a more general setting. We also refer to \cite{Kal1} for many interesting  properties of the sample paths of such processes. 

Let $B=(B(x))_{0\leq x \leq 1}$ be a real valued process with c\`adl\`ag sample paths, and which is continuous in probability.
We say that $B$  has exchangeable increments if for every $n\geq 2$, the sequence of its increments $B(k/n)-B((k-1)/n)$ for $ k=1, \ldots, n$,  is exchangeable, i.e. its distribution is invariant by permutations.  We further say that $B$ is a bridge provided that $B(0)=B(1)=0$ a.s.

According to Theorem 2.1 in \cite{Kal2}, any bridge with exchangeable increments can be expressed in the form
\begin{equation} \label{E:repbex}
B(x) = \sigma G(x) + \sum_{j=1}^{\infty}\beta_j({\mathbf 1}_{U_j\leq x} - x), \qquad x\in[0,1],
\end{equation}
where $\sigma$ is a nonnegative random variable, $G$ a Brownian bridge, 
$\bm{\beta}=(\beta_j)_{j\geq 1}$ a sequence of real random variables with $\sum_{j=1}^{\infty} \beta_j^2<\infty$ a.s.,
and $\bm{U}=(U_j)_{j\geq 1}$ a sequence of i.i.d. uniform random variables, such that $\sigma, G, \bm{\beta}$ and $\bm{U}$ are independent. More precisely, if we further assume that the sequence 
$(|\beta_j|)_{j\geq 1}$ is nonincreasing, which induces no loss of generality, then the series in \eqref{E:repbex} converges a.s. uniformly on $[0,1]$.

One calls $\sigma, \bm{\beta}$ the canonical representation of $B$. Roughly speaking, \eqref{E:repbex} shows that the continuous part of $B$ is a mixture of Brownian bridges (parametrized by the standard deviation), with mixture weights given by the random variable $\sigma$, and $\bm{\beta}$ describes the sequence of the jumps of $B$, each of them taking place uniformly at random on $[0,1]$ and independently of the others. The laws of $\sigma$ and of $\bm{\beta}$ then entirely determine that of $B$.

The next lemma plays a key role in the proof of Theorem \ref{T1}; it states two criteria  that are  tailored for our purposes, for the convergence of a sequence of bridges with exchangeable increments in terms of the canonical representations.
The first part is a special case of Theorem 2.3 in \cite{Kal2}. The second part can be seen as an immediate consequence of the first and the well-known facts  that Skorokhod's topology is metrizable and that convergence of a sequence of functions in ${\mathbb D}$ to a continuous limit is equivalent to convergence for the supremum distance (see, e.g. Section VI.1 in \cite{JS}); it can also be checked by direct calculation. 

\begin{lemma}\label{L1} For each $n\geq 1$, let $B_n$ denote a bridge with exchangeable increments and canonical representation $\sigma_n=0$ and $\bm{\beta}_n=(\beta_{n,j})_{j\geq 1}$. 
\begin{enumerate}
\item[(i)] Suppose that
$$
\lim_{n\to \infty} \sup_{j\geq 1} |\beta_{n,j}|=0 \text{  in probability,}$$
and that 
$$ \sum_{j=1}^{\infty} \beta_{n,j}^2 \ \Longrightarrow \ \sigma^2 $$
 for some  random variable $\sigma\geq 0$. 
Then there is the convergence in distribution
$$B_n \ \Longrightarrow \ \sigma G,$$
where $G$ is a standard Brownian bridge. 
\item[(ii)] If
$$\lim_{n\to \infty} \sum_{j=1}^{\infty} \beta_{n,j}^2=0\quad \text{in probability},$$
 then 
$$\lim_{n\to \infty} \sup_{0\leq x \leq 1}|B_n(x)| = 0 \quad \text{in probability}.$$ 
\end{enumerate}

\end{lemma}

\subsection{Asymptotic behavior of occurrences numbers}
Recall that the reinforcement parameter $p\in(0,1)$ in Simon's algorithm is fixed;  for the sake of simplicity,  it will be omitted from several notations even though it always plays an important role. 

For every $j\in \NN$, we set
$$N_j(n)=\# \{k\leq n: \hat U_k=U_j\}, \qquad n\geq 1,$$
that is $N_j(n)$ is the number of occurrences of the variable $U_j$ up to the $n$-th step of the algorithm.
Plainly $N_j(n)=0$ if and only if the number of innovations up to  the $n$-th step is less than $j$, i.e. ${\mathrm i}(n)<j$. 

The starting point of our analysis is that the reinforced empirical process   can be
expressed in the form
\begin{equation}\label{E:GN}
\hat G_n(x)= \frac{1}{\sqrt n} \sum_{j=1}^{\infty} N_j(n) ({\mathbf 1}_{U_j\leq x}-x), \quad x\in[0,1].
\end{equation}
Hence $\hat G_n$ is a bridge with exchangeable increments, with canonical representation $0$ and $\bm{\beta}_n=(\beta_{n,j})_{j\geq 1}$, where 
$\beta_{n,j}=N_j(n)/\sqrt n$. 
We aim to determine its asymptotic behavior as $n\to \infty$ by applying Lemma \ref{L1}. In this direction, the interpretation of Simon's algorithm as a Bernoulli bound percolation on a random recursive tree, as it has been sketched in the Introduction,  enables us to lift from
\cite{BauBer} the following result about the asymptotic behavior of mesoscopic clusters. 

\begin{lemma}\label{L4} The limit
\begin{equation}\label{E:surcrit}
\lim_{n\to \infty} n^{-p} N_j(n) \coloneqq X^{(p)}_j 
\end{equation}
exists a.s. for every $j\geq 1$. For $p>1/2$, there is furthermore the identity
$$\EE\left( \sum_{j=1}^{\infty} (X^{(p)}_j)^2\right) = \frac{1}{(2p-1)\Gamma(2p)}.$$
\end{lemma}
\begin{proof} Simon's algorithm induces a natural partition  $\NN=\bigsqcup_{j\geq 1} \Pi_j$ of the set of positive integers into blocks  $\Pi_j=\{k\in\NN: \hat U_k=U_j\}$ which we can see as the result of a Bernoulli bond percolation with parameter $p$ on the (infinite) random recursive tree.
In this setting, we have $N_j(n)=\#(\Pi_j\cap\{1, \ldots, n\})$, and the first claim of the statement has been observed in Section 3.2 of \cite{BauBer}, right after the proof of Lemma 3.3 there\footnote{ The reinforcement parameter $p$ here corresponds to $\e^{-t}$ in \cite{BauBer}.}. Moreover Equation (3.4) there shows that for every
$q>1/p$, there is the identity
$$\EE\left( \sum_{j=1}^{\infty} (X^{(p)}_j)^q\right) = \frac{\Gamma(q)}{\Gamma(pq)} + \frac{q(1-p)\Gamma(q)}{(pq-1)\Gamma(pq)}.$$
Specializing this for $q=2$ yields our second claim. 
\end{proof}

We write ${\mathbf X}^{(p)}=(X^{(p)}_j)_{j\geq 1}$, where the $X^{(p)}_j$ are defined by \eqref{E:surcrit}. It is known that $X^{(p)}_1$ has the Mittag-Leffler distribution with parameter $p$ (see Theorem 3.1 in \cite{BauBer} and also \cite{Moehle}); nonetheless the law of the whole sequence ${\mathbf X}^{(p)}$ does not seem to have any simple expression (see Proposition 3.7 in \cite{BauBer}).

When $p>1/2$,
Lemma \ref{L4} enables us to view   ${\mathbf X}^{(p)}$ as a random variable with values in the space $\ell^2(\NN)$
of square summable series, and this leads us to the following definition of the process $B^{(p)}$ that appears as a limit in Theorem \ref{T1}(iii).

\begin{definition}\label{D1} For $p>1/2$, we define $B^{(p)}=(B^{(p)}(x))_{0\leq x \leq 1}$ as the bridge with exchangeable increments with canonical representation $0$ and ${\mathbf X}^{(p)}$. That is
$$B^{(p)}(x) = \sum_{j=1}^{\infty}X^{(p)}_j({\mathbf 1}_{U_j\leq x} - x), \qquad x\in[0,1],$$
where ${\mathbf U}=(U_j)_{j\geq 1}$ is a sequence of i.i.d. uniform variables, independent of ${\mathbf X}^{(p)}$. 

\end{definition}

We conclude this section recalling the key result of Simon \cite{Simon} about the asymptotic frequency of microscopic percolation clusters. Note that the number of  innovations
up to the $n$-th step is approximately $(1-p)n$ for $n\gg 1$. 
\begin{lemma}\label{L2}
For each $k\geq 1$, write 
$$C_k(n) = \frac{1}{(1-p)n} \# \{j\geq 1: N_j(n)=k\},$$
for the number of variables $U_j$ which have occurred exactly $k$ times 
at the $n$-th step of Simon's algorithm.
Then
\begin{equation}\label{E:YuleSimon}
\lim_{n\to \infty} \frac{C_k(n)}{(1-p)n} = \frac{1}{p} \mathrm{B}(k, 1+1/p) \qquad \text{in probability for every }k\in\NN,
\end{equation}
where $\mathrm{B}$ denotes the beta function.
\end{lemma}
The right-hand side of \eqref{E:YuleSimon} is a probability measure on $\NN$ which is known as the Yule-Simon distribution with parameter $1/p$. 
 Actually, it is only proved in \cite{Simon} that 
 $$\EE (C_k(n))\sim \frac{1-p}{p} \mathrm{B}(k, 1+1/p)n \qquad\text{ as } n\to \infty;$$ nonetheless the stronger statement \eqref{E:YuleSimon} is known to hold; see e.g. Section 3.1 and more specifically Equation (3.10) in \cite{PPS}. 

\subsection{A first moment calculation}
Recall that we want to apply Lemma \ref{L1} to investigate the asymptotic behavior of  reinforced empirical processes. In this direction, \eqref{E:GN} incites us to introduce for every $n\geq 1$
$${\mathcal S}^2(n)=\sum_{j=1}^{\infty} N_j(n)^2.$$
The proof of Theorem \ref{T1} will use the following explicit calculation for the expectation of this quantity, which already points at the same direction as \eqref{E:expt}. 

\begin{lemma}\label{L3} For every $n\geq 1$, we have 
$$\EE({\mathcal S}^2(n)) =\frac{\Gamma(n+2p)}{\Gamma(n)}\sum_{i=1}^n \frac{\Gamma(i)}{\Gamma(i+2p)} .$$
As a consequence, we have as $n\to \infty$ that 
$$\EE({\mathcal S}^2(n))\sim\left\{ 
\begin{matrix}n/(1-2p) &\text{ if }p<1/2,\\
n\log n &\text{ if }p=1/2,\\ ((2p-1)\Gamma(2p))^{-1} n^{2p} &\text{ if }p>1/2.\\
\end{matrix}
\right.$$
\end{lemma}
\begin{proof}
Write ${\mathcal F}_n$ for the sigma-field generated by $((\varepsilon_i,v(i)): 2\leq i \leq n)$.
Plainly, $\sum_{j\geq 1} N_j(n)=n$, and we see from the very definition of Simon's algorithm that
\begin{align*}
\EE({\mathcal S}^2(n+1)\mid {\mathcal F}_n)&= {\mathcal S}^2(n) + p\left(\frac{1}{n} \sum_{j\geq 1} N_j(n)(2N_j(n)+1) \right) + (1-p)\\
&=(1+2p/n){\mathcal S}^2(n) + 1.
\end{align*}
This yields the  recurrence equation for the first moments
$$\EE({\mathcal S}^2(n+1))=(1+2p/n)\EE({\mathcal S}^2(n))+1.$$

To solve the latter, we set $a(n)=\Gamma(n+2p)/\Gamma(n)$, so that
$a(n+1)/a(n)= 1+2p/n$, and then
$$\EE({\mathcal S}^2(n+1))/a(n+1)=\EE({\mathcal S}^2(n))/a(n)+1/a(n+1).$$
Since ${\mathcal S}^2(1)=1$, we arrive at 
$$\EE({\mathcal S}^2(n))= a(n)\sum_{i=1}^{n}\frac{1}{a(i)},$$
which is the identity of our claim.

In turn, the estimate as $n\to \infty$ in the statement follows immediately from the facts that $\Gamma(n+2p)/\Gamma(n) \sim n^{2p}$, 
and that when $p>1/2$, one has
\begin{align}\label{E:sumbeta}
\sum_{i=1}^{\infty} \frac{\Gamma(i)}{\Gamma(i+2p)}&= \frac{1}{\Gamma(2p)} \sum_{i=1}^{\infty} {\mathrm B}(i,2p)\nonumber \\
&= \frac{1}{\Gamma(2p)}\int_0^1 \left(\sum_{i=1}^{\infty} x^{i-1}\right)(1-x)^{2p-1} \dd x \nonumber \\
&= \frac{1}{\Gamma(2p)}\int_0^1 (1-x)^{2p-2}\dd x \nonumber \\
&= \frac{1}{(2p-1)\Gamma(2p)}.
\end{align}
The proof is now complete. 
\end{proof}

As a first application, we establish the reinforced version of the Glivenko-Cantelli theorem.
\begin{proof}[Proof of Proposition \ref{P:GlivCant}]
The proof is classically reduced to establishing the following reinforced version of the strong law of large numbers, 
\begin{equation}\label{E:LLN}
\lim_{n\to \infty} \frac{1}{n} \sum_{i=1}^n {\mathbf 1}_{\hat U_i\leq x}=x \qquad\text{ a.s. for each }x\in[0,1].
\end{equation}
Indeed, the almost sure convergence in \eqref{E:LLN}  holds simultaneously for all  dyadic rational numbers, and uniform convergence on $[0,1]$ then can be derived  by a monotonicity argument \textit{\`a la} Dini.

So fix $x\in[0,1]$ and set 
$$\Sigma(n)= \sum_{i=1}^n {\mathbf 1}_{\hat U_i\leq x}= \sum_{j=1}^{\infty} N_j(n){\mathbf 1}_{U_j\leq x} .$$
Clearly, $\EE(\Sigma(n))=nx$, and, by conditioning on ${\mathcal F}_n$, we get 
$$\mathrm{Var}(\Sigma(n)) = (x-x^2)\EE({\mathcal S}^2(n)).$$
From Lemma \ref{L3} and  Chebychev's inequality, we now see that we can choose $r>1$ sufficiently large such that 
$$\sum_{k=1}^{\infty} \PP(|\Sigma(k^r)-k^rx|>k^{r-1})<\infty.$$
One concludes from the Borel-Cantelli lemma that \eqref{E:LLN} holds along the subsequence $n=k^r$, and the general case follows by another argument of monotonicity.
\end{proof}

\section{Proof of Theorem \ref{T1}}
As its title indicates, the purpose of this section is to establish Theorem \ref{T1} in each of the three regimes.

\subsection{Subcritical regime $p<1/2$}
\textit{Throughout this section, we assume that the reinforcement parameter satisfies $p<1/2$.}
Our approach in the subcritical regime relies on  the following strengthening of Lemma
\ref{L2} (recall the notation there).

\begin{lemma}\label{L5} Define for every $i\geq 1$
$$c_i^{(p)}=
\frac{(1-p) }{p} \mathrm{B}(i, 1+1/p).$$ Then we have
$$\lim_{n\to \infty}  \EE\left( \sum_{i=1}^{\infty} i^2 \left| \frac{C_i(n)}{n}
-c_i^{(p)}\right| \right) =0.$$
\end{lemma}
\begin{proof} For each $n=1,2, \ldots$,
write ${\mathbf C}(n)= (C_i(n))_{i\geq 1}$  and view ${\mathbf C}(n)$
as a function on the space $\Omega\times \NN$
endowed with the product measure $\PP\otimes \#^2$, where $\#^2$ stands for the measure on $\NN$ which assigns mass $i^2$ to every $i\in\NN$. 
Consider an arbitrary subsequence excerpt from $({\mathbf C}(n))_{n\geq 1}$.
From Lemma \ref{L2} and an argument of diagonal extraction, we can construct a further subsequence, say indexed by $\ell(n)$ for $n=1,2, \ldots$, such that 
\begin{equation}\label{E:a.e.}
\lim_{n\to \infty} \frac{{\mathbf C}(\ell(n))}{\ell(n)} = {\mathbf c}^{(p)} \qquad (\PP\otimes \#^2)
\text{-almost everywhere,}
\end{equation}
where
${\mathbf c}^{(p)}=(c_i^{(p)})_{i\geq 1}$.

On the one hand, we observe that
\begin{align*}\sum_{i=1}^{\infty} i^2\mathrm{B}(i, 1+1/p) &=\int_0^1 \left(\sum_{i=1}^{\infty} i^2 x^{i-1}\right)(1-x)^{1/p}\dd x\\
&=\int_0^1(1+x)(1-x)^{-3+1/p}\dd x\\
&=\frac{p}{(1-p)(1-2p)},
\end{align*}
so that
\begin{equation}\label{E:sumc}
\sum_{i=1}^{\infty} i^2 c_i^{(p)}= \frac{1}{1-2p}. 
\end{equation}

 On the other hand, we  note the basic identity
\begin{equation}\label{E:basic}
{\mathcal S}^2(n) = \sum_{j=1}^{\infty} N_j(n)^2= \sum_{i=1}^{\infty} i^2 C_i(n).
\end{equation}
Since $\Gamma(n+2p)/\Gamma(n)\sim n^{2p}$ and $2p<1$, 
we see from Lemma \ref{L3} and \eqref{E:basic} that
$$
\lim_{n\to \infty} \EE\left(\sum_{i=1}^{\infty} i^2 \frac{C_i(n)}{n}\right)= \frac{1}{1-2p}. 
$$
Thanks to \eqref{E:sumc}, we deduce from the Vitali-Scheff\'e  theorem (see, e.g. Theorem 2.8.9 in \cite{Bogachev}) that the convergence 
\eqref{E:a.e.} also holds in $L^1(\PP\otimes \#^2)$, that is
$$\lim_{n\to \infty}  \EE\left( \sum_{i=1}^{\infty} i^2 \left| \frac{C_i(\ell(n))}{\ell(n)}
-c_i^{(p)}\right| \right) =0.$$
Since the convergence above holds for any (initial) subsequence, our claim is proven.
\end{proof}
We next point at
 the following consequence of 
Lemma \ref{L5}.
\begin{corollary}\label{C1} We have
$$\lim_{n\to \infty} \frac{{\mathcal S}^2(n)}{n}=\frac{1}{1-2p} \qquad \text{in }L^1(\PP),$$
and 
$$
\lim_{n\to \infty} \sup_{j\geq 1} \frac{N_j(n)}{\sqrt n} =0\qquad \text{in probability}.$$
\end{corollary}
\begin{proof}
Observe from \eqref{E:sumc},  \eqref{E:basic}, and  the triangle inequality that
 \begin{align*}
 \EE\left(\left| \frac{{\mathcal S}^2(n)}{n}-\frac{1}{1-2p}\right| \right) &=
 \EE\left(\left|  \sum_{i=1}^{\infty} i^2 \frac{C_i(n)}{n}
-\sum_{i=1}^{\infty} i^2 c_i^{(p)}\right| \right)  \\
 &\leq 
 \EE\left( \sum_{i=1}^{\infty} i^2 \left| \frac{C_i(n)}{n}
-c_i^{(p)}\right| \right).
\end{align*}
Our first assertion thus follows from Lemma \ref{L5}. 

For the second assertion, observe that
$$\sup_{j\geq 1} N_j(n)=\sup\{i\geq 1: C_i(n)\geq 1\}.$$
We then have for every  $\eta >0$ arbitrarily small
\begin{align*}
\PP\left (\sup_{j\geq 1}N_j(n)>\eta \sqrt n \right) &=\PP(\exists i\geq \eta \sqrt n: C_i(n)\geq 1)\\
&\leq \frac{1}{\eta^2 n} \EE\left(\sum_{i\geq \eta \sqrt n} i^2 C_i(n)\right).
\end{align*}
It follows from Lemma \ref{L5} that the right-hand side converges to $0$ as $n\to \infty$, and the proof is now complete. 
\end{proof}
Theorem \ref{T1}(i) now derives immediately from \eqref{E:GN}, Lemma \ref{L1}(i) and Corollary \ref{C1} by  setting 
$\beta_{n,j}=N_j(n)/\sqrt n$ for every $j\geq 1$.

\subsection{Critical regime $p=1/2$}
\textit{Throughout this section, we assume that the reinforcement parameter is $p=1/2$.}
Recall from Lemma \ref{L3} that 
$\EE\left( {\mathcal S}^2(n) \right )\sim n \log n$ as $n\to \infty.$
We establish now a stronger version of this estimate.
\begin{lemma}\label{L6} One has
$$\lim_{n\to \infty} \frac{{\mathcal S}^2(n)}{n\log n} =1\qquad \text{almost surely}.$$
\end{lemma}

\begin{proof} It has been observed in \cite{BeUni} that, in the study of reinforcement induced by Simon's algorithm,  it may be convenient  to perform a time-substitution based on a Yule process. We shall use this idea here again,
and introduce a standard Yule process  $Y=(Y_t)_{t\geq 0}$, which we further assume to be independent of the preceding variables.
Recall that $Y$ is a pure birth process in continuous time started from $Y_0=1$ and  with birth rate $n$ from any state $n\geq 1$; in particular, for every function $f: \NN\to \RR$, say such that $f(n)=O(n^r)$ for some $r >0$,  the process
$$f(Y_t)-\int_0^t \left( f(Y_s+1)-f(Y_s)\right) Y_s \dd s$$
is a martingale.

Consider the time changed process ${\mathcal S}^2\circ Y$. Applying the observation above to $f= {\mathcal S}^2$ and then projecting on the natural filtration of  ${\mathcal S}^2\circ Y$, the same calculation as in the proof of Lemma \ref{L3} show that 
\begin{align*} &{\mathcal S}^2(Y_t) - \int_0^t \left(\frac{1}{2} +  \sum_{i=1}^{\infty} \frac{N_i(Y_s)}{2Y_s} (2N_i(Y_s)+1) \right)Y_s \dd s
 \\
 &= {\mathcal S}^2(Y_t) -\int_0^t({\mathcal S}^2(Y_s)+Y_s)\dd s
 \end{align*}
 is a  martingale. By elementary stochastic calculus, the same holds for
 $$M_t= \e^{-t} {\mathcal S}^2(Y_t) -\int_0^t \e^{-s}Y_s \dd s.$$
 
 We shall now show that $M$ is bounded in $L^2(\PP)$ by checking that its quadratic variation $[M]_{\infty}$
has a finite expectation. Plainly, $M$ is purely discontinuous;
its jumps can arise either due to an innovation event (whose instantaneous rate at time $t$ equals $\frac{1}{2}Y_{t-}$), and then $\Delta M_t=M_t-M_{t-}=\e^{-t}$, or by a repetition
of the $j$-th item for some $j\geq 1$ (whose instantaneous rate at time $t$ equals $\frac{1}{2}N_j(Y_{t-}))$, and then $\Delta M_t=\e^{-t}(2N_j(Y_{t-})+1)$. 
 We thus find by a standard calculation of compensation that
 \begin{align*}\EE([M]_{\infty})&=\EE\left(\sum_{t>0} (\Delta M_t)^2\right) \\
& = \EE\left(\int_0^{\infty} \e^{-2t} \left(\frac{1}{2} Y_t + \frac{1}{2}\sum_{j\geq 1} N_j(Y_t) (2N_j(Y_t)+1)^2\right)\dd t
 \right) \\
 & = \int_0^{\infty}  \EE\left( Y_t + 2 \sum_{j\geq 1} (N_j(Y_t)^3 + N_j(Y_t)^2)
 \right) \e^{-2t}  \dd t.
\end{align*}

First, recall that $Y_t$ has the geometric distribution with parameter $\e^{-t}$, in particular $\int_0^{\infty}  \EE( Y_t )\e^{-2t}\dd t = 1$. 
Second, $\sum_{j\geq 1} N_j(Y_t)^2 = {\mathcal S}^2(Y_t)$, and since $\EE({\mathcal S}^2(n))\sim n\log n$ (see Lemma \ref{L3}) and the processes $S$ and $Y$ are independent, we have also 
$$ \int_0^{\infty}  \EE\left( \sum_{j\geq 1}  N_j(Y_t)^2\right) \e^{-2t}  \dd t <\infty.$$
Third, consider  $T(Y_t)=\sum_{j\geq 1} N_j(Y_t)^3$. By calculations similar to those for $M_t$, one sees that
the process
$$\e^{-3t/2} T(Y_t) - \int_0^t \e^{-3s/2} (Y_s + {\mathcal S}^2(Y_s)) \dd s, \qquad t\geq 0$$
is a local martingale. Just as above, one readily checks that 
$$\int_0^{\infty} \e^{-3s/2} \EE(Y_s + {\mathcal S}^2(Y_s)) \dd s<\infty,$$
and hence $\EE(T(Y_t))=O(\e^{3t/2})$. As a consequence,
$$ \int_0^{\infty}  \EE\left( \sum_{j\geq 1}  N_j(Y_t)^3\right) \e^{-2t}  \dd t <\infty,$$
and putting the pieces together, we have checked that $\EE([M]_{\infty})<\infty$.

We now know that $\lim_{t\to \infty} M_t=M_{\infty}$ a.s. and in $L^2(\PP)$, and recall the classical feature that $\lim_{t\to \infty} \e^{-t}Y_t =W$ a.s., where $W$ has the standard  exponential distribution. In particular
$\int_0^t \e^{-s} Y_s \dd s \sim tW$ as $t\to \infty$, so that
$${\mathcal S}^2(Y_t) =  t\e^{t} W + o(\e^t),\qquad \text{a.s.}$$
Using again $Y_t= \e^t W+ o(\e^t)$, we conclude that
${\mathcal S}^2(n) = n\log n + O(n)$ a.s., which implies our claim. 
\end{proof}
 \begin{remark} The first part of Corollary \ref{C1} and Lemma \ref{L6}  seem to be of the same nature. Actually, one can also establish the former by adapting the proof of the latter,
 therefore circumventing the appeal to Lemma \ref{L2}. There is  nonetheless a fundamental difference between these two results: although the microscopic clusters (i.e. of size $O(1)$) determine the asymptotic behavior of  ${\mathcal S}^2(n)$ in the sub-critical case, they have no impact in the critical case as it is seen from Lemma \ref{L2}.
 \end{remark}

Thanks to Lemmas \ref{L1}(ii) and \ref{L6}, the following statement is the final piece of the proof of Theorem \ref{T1}(ii). 
\begin{lemma}\label{L7} One has
$$
\lim_{n\to \infty} \sup_{j\geq 1} \frac{N_j(n)}{\sqrt {n\log n}}=0\qquad \text{in probability}.$$
\end{lemma} 
\begin{proof} We shall show that there is some numerical constant $b$ such that
\begin{equation}\label{E:boundm}
\EE(N_j(n)^3)\leq b(n/j)^{3/2} \qquad \text{for all }j,n\geq 1
\end{equation}
Then, by Markov's inequality, we have that  for any $\eta>0$
$$\PP\left(N_j(n)>\sqrt{\eta n\log n}\right) \leq b(\eta n\log n)^{-3/2}  (n/j)^{3/2},$$
and by the union bound
$$\PP\left(\exists j\leq n: N_j(n)>\sqrt{\eta n\log n}\right) \leq b(\eta \log n)^{-3/2}\sum_{j\geq 1}j^{-3/2},$$
which proves our claim.

For $i=1,2,3$, set $a_i(n)=\Gamma(n+i/2)/\Gamma(n)$, so $a_i(n) \sim n^{i/2}$ and actually $a_2(n)= n$.
Recall that ${\mathrm i}(n)$ denotes the number of innovations up to the $n$-step of Simon's algorithm.
Take any $j\geq 1$ and, just as in the proof of Lemma \ref{L3}, observe that on the event ${\mathrm i}(n)\geq j$, one has
\begin{align*}
\EE\left(\frac{N_{j}(n+1)}{ a_1(n+1)} \mid {\mathcal F}_n\right) &= \frac{N_{j}(n)}{ a_1(n)},\\
\EE\left(\frac{N_{j}(n+1)^2}{ a_2(n+1) }\mid {\mathcal F}_n\right) &= \frac{N_{j}(n)^2}{ a_2(n)} +\frac{N_j(n)}{2na_2(n+1)},\\
\EE\left(\frac{N_{j}(n+1)^3}{ a_3(n+1)} \mid {\mathcal F}_n\right) &= \frac{N_{j}(n)^3}{ a_3(n)} + \frac{3N_{j}(n)^2}{ 2na_3(n+1)}+ \frac{N_{j}(n)}{ 2na_3(n+1)}.\\
\end{align*}
The trivial bound ${\mathrm i}(j)\leq j$ then yields for any $n\geq j$
$$\EE(N_j(n))\leq a_1(n)/a_1(j) \leq b_1\sqrt{n/j}.$$
Then we have 
$$\EE(N_j(n)^2)\leq \frac{a_2(n)}{a_2(j)} + \frac{b_1}{2na_2(n)} \sum_{k=j}^n \sqrt{k/j}\leq b_2 n/j,$$
and finally also
$$\EE(N_j(n)^3)\leq \frac{a_3(n)}{a_3(j)}   +
 \frac{3b_2}{2na_3(n)} \sum_{k=j}^n k/j+
 \frac{b_1}{2na_3(n)} \sum_{k=j}^n \sqrt{k/j}\leq b_3(n/j)^{3/2},
$$
where $b_1, b_2$ and $b_3$ are numerical constants. 
This establishes \eqref{E:boundm} and completes the proof.
\end{proof}

\subsection{Supercritical regime $p>1/2$}
\textit{Throughout this section, we assume that the reinforcement parameter satisfies $p>1/2$.}
We first point at  the following strengthening of Lemma \ref{L4} (in particular, recall the notation \eqref{E:surcrit} there).
\begin{corollary}\label{C4} We have
$$ \lim_{n\to \infty} \EE\left( \sum_{j=1}^{\infty} \left|\frac{ N_j(n)}{n^p} - X^{(p)}_j\right|^2\right)=0.$$
\end{corollary}
This result has been already observed by Businger, see Equation (6) in \cite{Businger}. For the sake of completeness, we present here an alternative and shorter proof along the same line as for Lemma \ref{L5}. 

\begin{proof} 
We view ${\mathbf X}^{(p)}=(X^{(p)}_j)_{j\geq 1}$ and ${\mathbf N}(n)=(N_j(n))_{j\geq 1}$ for each $n\geq 1$ as functions on the space $\Omega\times \NN$
endowed with the product measure $\PP\otimes \#$, where $ \#$ denotes the counting measure on $\NN$. Since we already know from Lemma
\ref{L4} that  $n^{-p} {\mathbf N}(n)$ converges as $n\to \infty$ to ${\mathbf X}^{(p)}$ almost everywhere, 
in order to establish our claim, it  suffices to verify that
$$ \lim_{n\to \infty} \EE\left( \sum_{j=1}^{\infty} \frac{ N_j(n)^2}{n^{2p}} \right)= \EE\left( \sum_{j=1}^{\infty} (X^{(p)}_j)^2 \right);$$
see e.g. Proposition 4.7.30 in \cite{Bogachev}. 

Recall from Lemma \ref{L3} that
$$\EE\left( \sum_{j=1}^{\infty} \frac{ N_j(n)^2}{n^{2p}} \right) = \frac{\Gamma(n+2p)}{n^{2p}\Gamma(n)}\sum_{i=1}^n \frac{\Gamma(i)}{\Gamma(i+2p)} .$$
On the one hand, we know that
$$\lim_{n\to \infty} \frac{\Gamma(n+2p)}{n^{2p}\Gamma(n)}=1,$$
and on the other hand, we recall from \eqref{E:sumbeta} that 
$$
\sum_{i=1}^{\infty} \frac{\Gamma(i)}{\Gamma(i+2p)}= \frac{1}{(2p-1)\Gamma(2p)}.
$$
We conclude from Lemma \ref{L4}  that indeed
$$ \lim_{n\to \infty} \EE\left( \sum_{j=1}^{\infty} \frac{ N_j(n)^2}{n^{2p}} \right)= \frac{1}{(2p-1)\Gamma(2p)}= \EE\left( \sum_{j=1}^{\infty} (X^{(p)}_j)^2 \right)$$
and the proof is complete.
\end{proof}
Theorem \ref{T1} can now be deduced  from \eqref{E:GN},  Lemma \ref{L1}(ii), and Corollary \ref{C4}.

\section{Relation to step reinforced random walks}
It is interesting to combine Donsker's Theorem with the continuous mapping theorem; notably considering the overall supremum of paths yields the well-known Kolmogorov-Smirnov test. In this direction,  linear mappings of the type $\omega \mapsto \int_{[0,1]} \omega(x) m(\dd x)$, where $m$ is some finite measure on $[0,1]$, are amongst the simplest functionals on ${\mathbb D}$. Writing 
$\bar m(x) =m((x,1])$ for the tail distribution function and $\mu=\int_{[0,1]} x m(\dd x)$ for the mean, this leads us to consider the variables 
$$ \xi_j= \bar m(U_j)-\mu \quad \text{and} \quad \hat \xi_j=\bar m(\hat U_j)-\mu \qquad \text{for }j\geq 1.$$
So $(\xi_j)_{j\geq 1}$ is an i.i.d. sequence and  $(\hat \xi_j)_{j\geq 1}$ can  be viewed as the reinforced sequence resulting from Simon's algorithm. 
All these variables have the same distribution, they are bounded and centered with variance 
$$\varsigma^2=\int_{[0,1]}\int_{[0,1]} (x\wedge y -xy)m(\dd x) m(\dd y)= \mathrm{Var}\left(\int_{[0,1]} G(x) m(\dd x)\right),$$
where $G=(G(x))_{0\leq x \leq 1}$ is a Brownian bridge.
In this setting, we have
$$ \int_{[0,1]} \hat G_n(x) m(\dd x) = \frac{\hat \xi_1+ \cdots + \hat \xi_n}{ \sqrt n}.$$

The process of the partial sums 
$$\hat S(n)= \hat \xi_1+ \cdots + \hat \xi_n, \qquad n\geq 0$$
is called a step reinforced random walk.
We now  immediately  deduce from Theorem \ref{T1} and the continuous mapping theorem that its asymptotic behavior is given by:
\begin{enumerate}
\item[(i)] if $p<1/2$, then 
$$n^{-1/2}\hat S(n)  \ \Longrightarrow \ {\mathcal N}(0,\varsigma^2/(1-2p));$$
\item[(ii)] if $p=1/2$, then  
$$(n\log n)^{-1/2} \hat S(n) \ \Longrightarrow \ {\mathcal N}(0,\varsigma^2);$$
\item[(iii)] if $p>1/2$, then 
$$\lim_{n\to \infty} n^{-p}\hat S(n) = \sum_{j\geq 1} \xi_j X^{(p)}_j\qquad \text{
in probability,}$$
 where ${\mathbf X}^{(p)}=(X^{(p)}_j)_{j\geq 1}$ has been defined in Lemma \ref{L4} and is independent of the $\xi_j$. 
\end{enumerate}

Although this argument only enables us to deal with real bounded random variables $\xi$, we stress that  more generally, the assertions (i), (ii) and (iii) still hold when the generic step $\xi$ is an arbitrary square integrable and centered variable in $\RR^d$ (for $d\geq 2$, $\varsigma^2$ is then of course the covariance matrix of $\xi$). Specifically, (i) follows from the invariance principle for step reinforced random walks (see Theorem 3.3 in \cite{BeUni}), whereas (iii) is Theorem 3.2 in the same work; see also \cite{Besca}. In the critical case $p=1/2$, (ii) can be deduced from the basic identity
$$\hat S(n) = \sum_{j=1}^{\infty} N_j(n) \xi_j,$$
 the L\'evy-Lindeberg theorem (see, e.g. Theorem 5.2 of Chapter VII in \cite{JS}), and Lemmas \ref{L6} and  \ref{L7}. 

In this vein, we mention that when $\xi$ has the Bernoulli distribution, (i-iii) are due originally to Heyde \cite{Heyde} in the setting of the so-called correlated Bernoulli processes, see also \cite{JaJaQi, WuQiYa, Zha2}. These results have also appeared more recently in the framework of the so-called elephant random walk, a random walk with memory which has been introduced by Sch\"utz and Trimper \cite{SchTr}. See notably \cite{BaurBer, Bercu, ColGavSch1, ColGavSch2}, and also \cite{Bercu2, Marco, Gonz, VG} and references therein for some further developments. We mention that K\"{u}rsten \cite{Kur} first pointed at the role of Bernoulli bond percolation on random recursive trees  in this framework, see also \cite{Businger}.
It is moreover interesting to recall that, for the elephant random random walk, Kubota and Takei \cite{KuMa} have established that the fluctuations corresponding to (iii) are Gaussian. Whether or not the same holds for general step reinforced random walks is still open; this also suggests that for $p>1/2$, Theorem \ref{T1} (iii) might be refined and yield a second order weak limit theorem involving again a Brownian bridge in the limit. 
 \bibliography{NoiseR.bib}

\end{document}